\RequirePackage{fix-cm}
\documentclass{svjour3}   
\def\makeheadbox{\relax}  

\smartqed  
\usepackage{graphicx}
\graphicspath{{images/}}

\usepackage{multicol,xcolor} 
\usepackage{anyfontsize}
\usepackage{enumitem,amssymb,amsmath}
\usepackage[all]{xy}
\usepackage{stmaryrd,dsfont}
\usepackage{float}

\newtheorem{thm}{Theorem}[section]
\newtheorem{Def}[thm]{Definition}
\newtheorem{ex}[thm]{Example}

\newtheorem{prop}[thm]{Proposition}
\newtheorem{cor}[thm]{Corollary}
\newtheorem{lem}[thm]{Lemma}

\newcommand{\Xk}{(X,\kappa)}
\newcommand{\Xp}{(X,p)}
\newcommand{\Yl}{(Y,\lambda)}

\newcommand{\kl}{(\kappa,\lambda)}
\newcommand{\Tn}{T:I^n\to X}
\newcommand{\piX}{{\rm{\Pi}}^\kappa_1(X,p)}
\newcommand{\clk}{Cl_\kappa}
\newcommand{\ink}{Int_\kappa}

\newcommand{\cat}[1]{\mathsf{#1}}

\begin{document}

\title{Digital Hurewicz Theorem and Digital Homology Theory 
}


\author{Samira Sahar Jamil         \and
        Danish Ali 
}


\institute{S. S. Jamil \at
			Department of Mathematics, University of Karachi, Karachi, Pakistan\\
              Department of Mathematical Sciences, Institute of Business Administration (IBA), Karachi, Pakistan               
\\
              \email{ssjamil@iba.edu.pk}           
           \and
           D. Ali \at
           Department of Mathematical Sciences, Institute of Business Administration (IBA), Karachi, Pakistan\\
}

\date{}

\maketitle
\begin{abstract}
In this paper, we develop homology groups for digital images based on cubical singular homology theory for topological spaces. Using this homology, we present digital Hurewicz theorem for the fundamental group of digital images. We also show that the homology functors developed in this paper satisfy properties that resemble the Eilenberg-Steenrod axioms of homology theory, in particular, the homotopy and the excision axioms. We finally define axioms of digital homology theory. 
\end{abstract}
\section{Introduction}\label{Sec_Intro}
Digital images can be considered as objects in $\mathbb{Z}^d$, where $d$ is 2 for two dimensional images and $d$ is 3 for three-dimensional images. 
Though these objects are discrete in nature, they model continuous objects of the real world.  Researchers are trying to understand whether or not digital images show similar properties as their continuous counterparts. 
The main motivation behind such studies is to develop a theory for digital images that is similar to the theory of topological spaces in classical topology. 
Due to discrete nature of digital images, it is difficult to get results that are analogous to those in classical topology. 
Nevertheless, great effort has been made by researchers to derive a theory that goes in line with general topology. 
Several notions that are well-studied in general topology and algebraic topology, have been developed for digital images, which include
continuity of functions \cite{Rosen86,Boxer94},
Jordan curve theorem \cite{Rosen79,Slapal04}, 
covering spaces \cite{Han08}, 
fundamental group \cite{Kong_89,Boxer_99},
homotopy (see \cite{Boxer06}, \cite{Boxer10}), 
homology groups \cite{Arslan_08,Boxer_homo_2011,Lee_11,Karaca_2012,Ege_2013,Ege_2014},
cohomology groups \cite{Ege_2013}, H-spaces \cite{Ege16} and fibrations \cite{Ege17}.

The idea of fundamental group was first introduced in the field of digital topology by Kong \cite{Kong_89}. 
Boxer \cite{Boxer_99} adopted a classical approach to define and study fundamental group, which was closer to the methods of algebraic topology. 
Simplicial homology groups were introduced in the field of digital topology by Arslan \textit{et al.} \cite{Arslan_08} and extended by Boxer \textit{et al.}  \cite{Boxer_homo_2011}. 
Eilenberg-Steenrod axioms for simplicial homology groups of digital images were investigated by Ege and Karaca in \cite{Ege_2013}, where it is shown that all these axioms hold in digital simplicial setting except for homotopy and excision axioms. 
They demonstrate using an example that Hurewicz theorem does not hold in case of digital simplicial homology groups. Relative homology groups of digital images in simplicial setting were studied in \cite{Ege_2014}.
Karaca and Ege \cite{Karaca_2012} developed the digital cubical homology groups in a similar way as the cubical homology groups of topological spaces in algebraic topology. 
Unlike the case of algebraic topology,  digital cubical homology groups are in general not isomorphic  to  digital simplicial homology groups studied in \cite{Boxer_homo_2011}. 
Furthermore, Mayer-Vietoris theorem fails for cubical homology on digital images, which is another contrast to the case of algebraic topology. 
In \cite{Lee_11}, singular homology group of digital images were developed.
 
We do not know of any homology theory on digital images that relates to the fundamental group of digital images developed in \cite{Boxer_99}, as in algebraic topology. 
This is the main motivation behind work done in this paper. 
This paper is organized as follows. We review some of the basic concepts of digital topology in Section \ref{Sec_Prelim}.
We develop homology groups of digital images based on cubical singular homology of topological spaces as given in \cite{Massey91} in Section \ref{Sec_CubSing}, and give some basic results including the functoriality, additivity and homotopy invariance of cubical singular homology groups. 
In Section \ref{Sec_Hure}, we show that the fundamental group for digital images (given by \cite{Boxer_99}) is related to our first homology group, and obtain a result that is analogous to Hurewicz theorem of algebraic topology.
In Section \ref{Sec_RelExc}, we prove a result for cubical singular homology on digital  images (Theorem \ref{Thm_Excision_nleq2}) similar to the excision theorem of algebraic topology except that our result holds only in dimensions less than 3.  
This result is then generalized and we call this generalization `Excision-like property'  for cubical singular homology on digital images (Theorem \ref{Thm_Excision_like}). 
Cubical singular homology groups satisfy properties that are much similar to the Eilenberg-Steenrod axioms of homology theory.
We define digital homology theory in Section \ref{Sec_DigHom}, the axioms of which can be regarded as digital version of Eilenberg-Steenrod axioms in algebraic topology. We also show that cubical singular homology is a digital homology theory. Throughout this paper, we consider finite binary digital images, though most of the results also hold for infinite case. 

\section{Preliminaries}\label{Sec_Prelim}
\subsection{Basic concepts of digital topology}
Let $\mathbb{Z}^d$ be the Cartesian product of $d$ copies of set of integers $\mathbb{Z}$, for a positive integer $d$. 
A digital image is a subset of $\mathbb{Z}^d$. 
A relation that is symmetric3.3 and irreflexive is called an \textit{adjacency relation}. In digital images, adjacency relations give a concept of proximity or closeness among its elements, which allows some constructions in digital images that closely resemble those in topology and algebraic topology. The adjacency relations on digital images used in this paper are defined below.
\begin{Def}{\rm{\cite{Boxer06}}}
Consider  a positive integer $l$, where $1 \leq l \leq d$. The points $p,q \in \mathbb{Z}^d$ are said to be $c_l$\textit{-adjacent} if they are different and there are at most $l$ coordinates of $p$ and $q$ that differ by one unit, while the rest of the coordinates are equal. 
\end{Def}
Usually the notation $c_l$ is replaced by number of points $\kappa$ that are $c_l$-adjacent to a point. For $\mathbb{Z}^2$, 
there are $4$ points that are $c_1$-adjacent to a point and there are 8 points that are $c_2$-adjacent to a point, thus $c_1=4$ and $c_2=8$. 
Two points that are $\kappa$-adjacent to each other, are said to be $\kappa$-neighbors of each other. For $a,b \in \mathbb{Z}$, $a<b$, a \textit{digital interval} denoted as $[a,b]_\mathbb{Z}$ is a set of integers from $a$ to $b$, including $a$ and $b$. 
The digital image $X\subseteq\mathbb{Z}^d$ equipped with adjacency relation $\kappa$ is represented by the ordered pair $(X,\kappa)$.
\begin{Def}{\rm{\cite{Kong_89}\cite{Boxer_99}}}
Let $\Xk$ and $\Yl$ be digital images.\begin{enumerate}[label={\rm{(\roman*)}}]
\item  The function $f:X\to Y$ is \textit{$\kl$-continuous} if for every pair of $\kappa$-adjacent points $x_0$ and $x_1$ in $X$, either the images $f(x_0)$ and $f(x_1)$ are equal or $\lambda$-adjacent. 
\item Digital image $\Xk$ is said to be \textit{$\kl$-homeomorphic} to $\Yl$ , if there is a $\kl$-continuous bijection $f:X\to Y$, which has a $(\lambda,\kappa)$-continuous inverse $f^{-1}:Y\to X$.
\item A \textit{$\kappa$-path} in $\Xk$ is a $(2,\kappa)$-continuous function $f:[0,m]_\mathbb{Z}\to X$. We say $f$ is $\kappa$-path of length $m$ from $f(0)$ to $f(m)$. 
For a given $\kappa$-path $f$ of length $m$, we define \textit{reverse $\kappa$-path} $\overline{f}:[0,m]_\mathbb{Z}\to X$ defined by $\overline{f}(t)=f(m-t)$. 
A \textit{$\kappa$-loop} is a $\kappa$-path $f:[0,m]_\mathbb{Z}\to X$, with $f(0)=f(m)$.
\item
A subset $A \subset X$ is \textit{$\kappa$-connected} if and only if for all $x,y \in A$, $x \neq y$, 
there is a $\kappa$-path from $x$ to $y$. 
A \textit{$\kappa$-component} of a digital image is the maximal $\kappa$-connected subset of the digital image. 
\end{enumerate}
\end{Def}
\begin{Def}
Consider digital images $(X,\kappa)$ and $(Y,\kappa)$ with $X,Y\subset\mathbb{Z}^d$. 
\begin{itemize}[label={\textbullet}]
\item We say that $\Xk$ is \textit{$\kappa$-connected} with $(Y,\kappa)$ if there is $x\in X$ and $y\in Y$ such that $x$ and $y$ are $\kappa$-adjacent in $\mathbb{Z}^d$.
\item If $(X,\kappa)$ is not $\kappa$-connected with $(Y,\kappa)$, we say that $(X,\kappa)$ is $\kappa$-disconnected with $(Y,\kappa)$.
\end{itemize}
\end{Def}
\begin{prop}{\rm{\cite{Boxer_99}}}
If  $f:X\to Y$  is a $\kl$-continuous function, with $A\subset X$ a $\kappa$-connected subset, then $f(A)$ is $\lambda$-connected in $Y$.
\end{prop}
\begin{Def}{\rm\cite{Boxer_99}}
~\begin{enumerate}[label={\rm{(\roman*)}}]
\item Let $f,g:X\to Y$ be $\kl$-continuous functions. Suppose there is a positive integer $m$ and a function $H:[0,m]_\mathbb{Z}\times X\to Y$ such that:
\begin{itemize}[label=\small{\textbullet}]
\item for all $x\in X$, $H(0,x)=f(x)$ and $H(m,x)=g(x)$,
\item for all $x\in X$, the function $H_x:[0,m]_\mathbb{Z}\to Y$ defined by $H_x(t)=H(t,x)$ for all $t\in [0,m]_\mathbb{Z}$ is $(2,\lambda)$-continuous,
\item for all $t\in [0,m]_\mathbb{Z}$, the function $H_t:X\to Y$ defined by $H_t(x)=H(t,x)$ for all $x\in X$ is $\kl$-continuous. 
\end{itemize}
Then $H$ is called \textit{$\kl$-homotopy} from $f$ to $g$ and $f$ and $g$ are said to be \textit{$\kl$-homotopic}, denoted as $f\simeq_{\kl} g$. If $g $ is a constant function, $H$ is a \textit{null-homotopy} and $f$ is \textit{null-homotopic}. 
\item Two digital images $\Xk$ and $\Yl$ are \textit{homotopically equivalent}, if there is a $\kl$-continuous function $f:X\to Y$ and $(\lambda,\kappa)$-continuous function $g:Y\to X$ such that $g\circ f\simeq_{\kl} 1_X$ and $f\circ g\simeq_{(\lambda,\kappa)} 1_Y$, where $1_X$ and $1_Y$ are identity functions on $X$ and $Y$, respectively.
\item Let $H:[0,m]_\mathbb{Z}\times [0,n]_\mathbb{Z}\to X$ be a homotopy between $\kappa$-paths $f,g:[0,n]_\mathbb{Z}\to X$ in $\Xk$. The homotopy $H$ is said to\textit{ hold the end-points fixed} if $f(0)=H(t,0)=g(0)$ and $f(n)=H(t,n)=g(n)$ for all $t\in [0,m]_\mathbb{Z}$.
\end{enumerate}
 
\end{Def} 

\subsection{Digital Fundamental group}\label{Sec_Fund}
The concept of fundamental group for digital images was first given by \cite{Kong_89}, but a more classical approach to define and study digital fundamental group was adopted by Boxer \cite{Boxer_99}. We briefly explain digital fundamental group as defined in the latter paper.
\begin{Def}{\rm\cite{Boxer_99}}
\begin{enumerate}[label={\rm{(\roman*)}}]
\item A pointed digital image is a pair $(X,p)$, where $X$ is a digital image and $p\in X$. A pointed digital image $\Xp$ can be represented as $(\Xp,\kappa)$, if one wishes to emphasize the adjacency relation of the digital image $X$.  
\item Let $f$ and $g$ be $\kappa$-paths of lengths $m_1$ and $m_2$, respectively, in the pointed digital image $\Xp$, such that $g$ starts where $f$ ends, \textit{i.e.} $f(m_1)=g(0)$. The `product'  $f\ast g$ of two paths  is defined as follows:
\[ (f\ast g)(t)=
\begin{cases} 	f(t), &\text{if }t\in[0,m_1]_\mathbb{Z} \\
				g(t-m_1), &\text{if }t\in[m_1,m_1+m_2]_\mathbb{Z}. \end{cases}
\]
\end{enumerate}
\end{Def}
The concept of \textit{trivial extension} allows stretching the domain of a loop, without changing its homotopy class and thus allows to compare homotopy properties of paths even when the cardinalities of their domain differ.
\begin{Def}{\rm\cite{Boxer_99}}
\begin{enumerate}[label={\rm{(\roman*)}}]
\item Let $f$ and $f'$ be $\kappa$-paths in a pointed digital image $\Xp$.
We say that  $f'$ is a trivial extension of $f$, if there exist sets of $\kappa$-paths $\{f_1,f_2,\ldots,f_k\}$ and $\{f'_1,f'_2,\ldots,f'_n\}$ in $X$ such that 
\begin{itemize}[label=\textbullet]
\item $0<k\leq n$
\item $f=f_1\ast f_2\ast\cdots\ast f_k$
\item $f'=f'_1\ast f'_2\ast\cdots\ast f'_n$
\item there are indices $1\leq i_1<i_2<\cdots <i_k\leq n$ such that:
\begin{itemize}[label=\scriptsize\textbullet]
\item $f'_{i_j}=f_j, 1\leq j\leq k$ and
\item $i\notin \{i_1,i_2,\ldots,i_k\}$ implies $f'_i$ is a constant $\kappa$-path.
\end{itemize}
\end{itemize}
\item Two $\kappa$-loops $f$ and $g$ with the same basepoint $p\in X$ \textit{belong to the same loop class}, if there exist trivial extensions of $f$ and $g$, which have homotopy between them that holds the end-points fixed. 
\end{enumerate}\end{Def}
\begin{Def}{\rm\cite{Boxer_99}} Let $\piX$ be the set of loop classes in $\Xp$ with basepoint $p$. Let $[ f]_{_{\rm\Pi}}$ denote the loop class of $\kappa$-loop $f$ in $\Xk$. The product operation $\ast$ defined as:
\[[ f]_{_{\rm\Pi}}\ast [ g]_{_{\rm\Pi}}=[ f\ast g]_{_{\rm\Pi}}\]
is well defined on $\piX$ as well as associative \cite{Boxer_99}. The loop class $[ c]_{_{\rm\Pi}}$ of the constant loop is identity in $\piX$ with respect to taking product. For every loop class $[ f]_{_{\rm\Pi}}$ the loop class $[ \overline{f}]_{_{\rm\Pi}}$, where $\overline{f}$ is the reverse path of $f$, is the inverse of $[ f]_{_{\rm\Pi}}$ with respect to taking product $\ast$. Thus $\piX$ is a group under $\ast$  and called the \textit{digital fundamental group} of the pointed digital image $\Xp$. 
\end{Def}

\section{Cubical Singular Homology on Digital images}\label{Sec_CubSing}
Consider digital interval $I=[0,1]_\mathbb{Z}$. Let $I^n$ be the Cartesian product of $n$ copies of $I$ for $n>0$. 
We shall consider $I^n$ as a digital image $(I^n,2n)$. By definition, $I^0$ is a digital image consisting of single point.
For an integer $n\geq0$, a \textit{digitally singular $n$-cube} or briefly a digital \textit{$n$-cube} in  $(X,\kappa)$ is a $(2n,\kappa)$-continuous map $T:I^n\to X$.
\\
For an integer $n\geq 0$, let $dQ_{n,\kappa}(X)$ denote the free Abelian group generated by the set of all digitally singular $n$-cubes in $(X,\kappa)$. 
We write $dQ_n(X)$ for $dQ_{n,\kappa}(X)$, when the adjacency relation is clear from the context. 
An element of $dQ_n(X)$ is a finite formal linear combination of digital $n$-cubes. 
The basis of the group $dQ_0(X)$ can be identified with $X$ itself, and one can denote the elements of $dQ_0(X)$ as $\sum_i m_ix_i$, where $x_i\in X$.
A digitally singular $n$-cube $T:I^n\to X$ is \textit{degenerate}
 if there is an integer $i$, $1\leq i\leq n$ such that $T(t_1,t_2,\ldots, t_n)$ does not depend on $t_i$. 
Let $dD_{n,\kappa}(X)$, or simply $dD_n(X)$, denote the subgroup of $dQ_n(X)$ generated by the set of all degenerate digitally singular $n$-cubes in $(X,\kappa)$. 
Let $dC_{n,\kappa}(X)$, or simply $dC_n(X)$, denote the quotient group $dQ_n(X)/dD_n(X)$. We say $dC_n(X)$ is  the group of \textit{digitally cubical singular $n$-chains} in $(X,\kappa)$ and the elements of $dC_n(X)$ are \textit{$n$-chains} in $\Xk$. For any digital image $X$, $dC_n(X)$ can be shown as free Abelian group generated by non-degenerate digital $n$-cubes in $X$.\\
We define faces of a digitally singular $n$-cube as follows:
For a digital  $n$-cube $T:I^n\to X$ and $i=1,2,\ldots,n$, we define digital $(n-1)$-cubes $A_iT,B_iT:I^{n-1}\to X$ as 
\begin{align*}
&A_iT(t_1,t_2,\ldots,t_{n-1})= T(t_1,t_2,\ldots,t_{i-1},0,t_i,\ldots,t_{n-1}), \\
 \text{and }~~~& B_iT(t_1,t_2,\ldots,t_{n-1})= T(t_1,t_2,\ldots,t_{i-1},1,t_i,\ldots,t_{n-1}).
\end{align*} 
$A_iT$ and $B_iT$ are called \textit{front $i$-face} and \textit{back $i$-face} of $T$, respectively.\\
We define the \textit{boundary operator} $\partial_n$ on the basis element of $dQ_n(X)$ as 
$\partial_n(T)=\sum_{i=1}^n (-1)^i (A_iT-B_iT)$ and extend it by linearity (see \cite{AT_Rotman}, 
for the definition of extension by linearity) to get the homomorphism $\partial_n:dQ_n(X)\to dQ_{n-1}(X)$, $n\geq 1$. 
One may write $\partial$ for $\partial_n$ if $n$ is clear from the context. 
For $n<0$, let $dQ_n(X)=dC_n(X)=0$ and for $n \leq 0$, let $\partial_n=0$.
It can be shown that $\partial_{n-1}\partial_n=0$, for all integers $n$ (see \cite{Massey91} for details). 
A \textit{cubical singular complex} of the digital image $\Xk$, denoted as $(C_{\bullet,\kappa}(X),\partial)$ 
or $(dC_\bullet(X),\partial)$,  is the following chain complex:
\begin{equation*}
\xymatrix{
\cdots \ar[r]^{\partial_{n+1}~~~~}  
& dC_n(X) \ar[r]^{\partial_{n}~~}
& dC_{n-1}(X) \ar[r]^{~~~\partial_{n-1}}
& \cdots 
}
\end{equation*}
Let $dZ_n(X)$ denote the kernel of $\partial_n$ and $dB_n(X)$ denote the image of $\partial_{n+1}$, for all integers $n$. 
The elements of $dZ_n(X)$ and $dB_n(X)$ are called \textit{$n$-cycles} and \textit{$n$-boundaries} of $\Xk$, respectively.
We define \textit{$n^{th}$ cubical singular homology group} of the digital image $(X,\kappa)$, 
as $dH_{n,\kappa}(X)=H_n(dC_\bullet,\partial)=dZ_n(X)/dB_n(X)$, for all non-negative integers $n$. 
If the adjacency relation $\kappa$ is clear from context, we shall simply write $dH_n(X)$ for $dH_{n,\kappa}(X)$.
\paragraph{\textbf{$\boldsymbol{\kappa}$-path and digital $\bf{1}$-cubes:} }
A digital $1$-cube $T:I\to X$ in a digital image $\Xk$ can be considered as a $\kappa$-path of length 1.
A $\kappa$-path $f$ of length $m$ can be ``subdivided" into smaller paths of length 1 or  digital $1$-cubes. 
For a $\kappa$-path $f$ of length $m$, we can associate an element $\sum_{j=1}^m f_j$ of $dQ_1(X)$ to $f$, where $f_j:I\to X$ as $f_j(t)=f(j+t-1)$. We say that the element $\sum_{j=1}^m f_j$ is \textit{subdivision} of $f$.
Following are some properties of
 subdivision  $\sum_{j=1}^m f_j$ of $f$:
\begin{enumerate}
\item $f_j$ are degenerate, whenever $f(j-1)=f(j)$
\item If $f$ is a non-constant path then $\sum_{j=1}^mf_j$ is not degenerate, and so $\sum_{j=1}^mf_j$ is a nontrivial element in $dC_1(X)$, where some $f_j$ might be 0 in $dC_1(X)$.
\item $\partial\left(\sum_{j=1}^mf_j\right)=f(m)-f(0)$.
\item If $f$ is a $\kappa$-loop then $\sum_{j=1}^mf_j$ is a $1$-cycle.
\end{enumerate}

\begin{prop}\label{Prop_kConnectedZ}
If $(X,\kappa)$ be a non-empty $\kappa$-connected digital image, then \mbox{$dH_0(X)\approx \mathbb{Z}$}.
\end{prop}
\begin{proof}
Consider the map $\varepsilon:dC_0(X)\to\mathbb{Z}$ defined as $\sum_i m_i x_i\mapsto \sum_i m_i$. Now for $\sum_i n_i T_i\in dC_1(X)$, we have $\varepsilon\circ\partial(\sum_i n_i T_i)=\varepsilon(\sum_i n_i(B_1T-A_1T))=\sum_i(n_i-n_i)=0$. 
Thus $dB_0(X)\subset ker(\varepsilon)$. The reverse relation also holds for the following reason. Consider $\sum_im_ix_i\in ker(\varepsilon)$. We have $\sum_im_i=0$. 
Consider $x\in X$ ($X$ is non-empty) and $\kappa$-paths $f_i$ ($X$ is $\kappa$-connected) from $x$ to $x_i$. These paths can be subdivided to form elements $\sum_j f_{ij}\in dC_1(X)$ for each $i$. 
It can be verified that $\partial(\sum_jf_{ij})=x_i-x$. Thus $\partial(\sum_{i,j}m_if_{ij})=\sum_i m_ix_i-(\sum_i m_i)x=\sum_i m_ix_i$, implying $\sum_i m_i x_i\in dB_0(X)$. 
From first isomorphism theorem of groups $dH_0(X)=dZ_0(X)/dB_0(X)=dC_0(X)/dB_0(X)\approx\mathbb{Z}$.
\qed
\end{proof}

\begin{prop}\label{Prop_DirectSum}
Let $\{X_\alpha\vert\alpha\in\Lambda\}$ be the set of $\kappa$-components of the digital image $(X,\kappa)$. Then $dH_n(X)\approx\bigoplus_\alpha dH_n(X_\alpha)$.
\end{prop}
\begin{proof}
The groups $dQ_n(X)$, $dD_n(X)$ and $dC_n(X)$ break up to $\bigoplus_\alpha dQ_n(X_\alpha)$,
$\bigoplus_\alpha dD_n(X_\alpha)$ and $\bigoplus_\alpha dC_n(X_\alpha)$, respectively, because, the image of each digital $n$-cube $T$ lies entirely in one $\kappa$-component of $\Xk$ (see Section \ref{Sec_Prelim}). 
We also have  $dZ_n(X)=\bigoplus_\alpha dZ_n(X_\alpha)$ and $dB_n(X)=\bigoplus_\alpha dB_n(X_\alpha)$, and hence $dH_n(X)=\bigoplus_\alpha dH_n(X_\alpha)$, because the boundary map $\partial_n:dC_n(X)\to dC_{n-1}(X)$ maps $dC_n(X_\alpha)$ to $dC_{n-1}(X_\alpha)$.
\qed
\end{proof}
\begin{prop}
For any digital image $(X,\kappa)$, $dH_0(X)$ is a free Abelian group with rank equal to the number of $\kappa$-components of $(X,\kappa)$.
\end{prop}
\begin{proof}
Follows from Propositions \ref{Prop_kConnectedZ} and \ref{Prop_DirectSum}.
\qed
\end{proof}

\begin{prop}\label{prop_funcHn}
The cubical singular homology group $dH_n(-)$ is a functor from $\cat{Dig}$ to $\cat{Ab}$.
\end{prop}
\begin{proof}
We define $dH_n(-)$ on morphisms of $\cat{Dig}$ as follows: Consider a $\kl$ -continuous function $f:X\to Y$ from digital image $\Xk$ to digital image $\Yl$. 
For a digital  $n$-cube $\Tn$ in $dQ_n(X)$, we have $f\circ T\in dQ_n(Y)$. We define functions $f_\#:dQ_n(X)\to dQ_n(Y)$ as $T\mapsto f\circ T$ and extending by linearity, for integers $n\geq 0$. 
Since $f_\#(T)$ is degenerate, if $T\in dD_n(X)$, the map $f_\#$ induces $f_\#:dC_n(X)\to dC_n(Y)$, for integers $n\geq 0$. 
It can be shown that $f_\#$ is a chain map that sends $n$-cycles to $n$-cycles and $n$-boundaries to $n$-boundaries, and therefore induces a map $f_\ast=dH_n(f):dH_n(X)\to dH_n(Y)$ defined as $[T]\mapsto[f_\#(T)]$.

Furthermore, it can be easily shown that for an identity map $id:X\to X$ the induced map  $id_\ast=dH_n(id):dH_n(X)\to dH_n(X)$ is an identity map. Also 
for functions $f:X\to Y$ and $g:Y\to Z$, which are $\kl$- and $(\lambda,\gamma)$-continuous, we have $(g\circ f)_\ast=g_\ast\circ f_\ast:dH_n(X)\to dH_n(Z)$, because $(g\circ f)_\#=g_\#\circ f_\#:dQ_n(X)\to dQ_n(Z)$.
\qed\end{proof}
The following can be easily proved.

\begin{prop}
Let $\Xk$ and $\Yl$ be $\kl$-homeomorphic digital images, then $dH_n(X)=dH_n(Y)$, for all $n$.
\end{prop}
\begin{prop}\label{prop_Dim}
If $X=\{x_0\}$ is a one-point digital image, then
\[
 dH_n(X) = 
  \begin{cases} 
   \mathbb{Z},  & \text{if }n=0 \\
   0,        & \text{otherwise.}
  \end{cases}
\]
\end{prop}

\begin{thm}\label{prop_homotopy}
Let $f,g:X\to Y$ be $\kl$-homotopic maps from digital image $\Xk$ to the digital image $\Yl$. 
Then $f$ and $g$ induce the same maps on homology group $dH_n(X)$, \textit{i.e.} $f_\ast=g_\ast$ .
\end{thm}
\begin{proof}
Let $F:[0,m]_\mathbb{Z}\times X\to Y$ be the homotopy from $f$ to $g$. The homotopy $F$ can be subdivided  into functions $F_j:I\times X\to Y$ defined as $F_j(t,x)=F(j+t-1,x)$ for $j\in [1,m]_\mathbb{Z}$. 
Observe that $F_1(0,x)=f(x)$ and $F_m(1,x)=g(x)$. 
In order to show that $f_\ast=g_\ast$, we follow the standard method of algebraic topology, which is, to construct a map $\Phi_n:dQ_n(X)\to dQ_{n+1}(Y)$ that contains similar information as the Homotopy $F$, and satisfies: \begin{equation}\label{eq_ch_homotopy}
g_\#-f_\#=\partial_{n+1}\Phi_n+\Phi_{n-1}\partial_n
\end{equation} 
Define $\Phi_n:dQ_n(X)\to dQ_{n+1}(Y)$ as $T\mapsto \sum_{j=1}^m F_j(id\times T)$ and extending by linearity, where $id:[0,1]_\mathbb{Z}\to [0,1]_\mathbb{Z}$ is identity function. 
We need to compute the boundary $\partial\Phi$ to verify eq. \ref{eq_ch_homotopy}. One can observe the following:
\begin{equation}
A_1\Phi_nT=f_\#(T)+\sum_{j=2}^mF_j(0, T)
\hspace{10pt}\text{ and }\hspace{10pt}
B_1\Phi_n(T)=\sum_{j=1}^{m-1} F_j(1, T)+g_\#(T) \label{eq1}
\end{equation}
\begin{equation}
A_i\Phi_n(T)=\Phi_{n-1}A_{i-1}T,          
\hspace{10pt}\text{ and }\hspace{10pt}
B_i\Phi_n(T)=\Phi_{n-1}B_{i-1}T, \hspace{10pt} i\in [2,n+1]_\mathbb{Z}
\label{eq3}
\end{equation}
\begin{equation}\label{eq5}
F_j(1, T)=F_{j+1}(0, T), \hspace{10pt} j\in [1,m-1]_\mathbb{Z}
\end{equation}
Using these equations we can calculate the boundary of $\Phi$:
\begin{align*}
&&\partial\Phi_n(T)=&\sum_{i=1}^{n+1}(-1)^i(A_i\Phi_n(T)-B_i\Phi_n(T))\\
& &=&g_\#(T)-f_\#(T)+\sum_{i=2}^{n+1}(-1)^i(A_i\Phi_n(T)-B_i\Phi_n(T))&\\&&&\hspace{50pt}\text{using eqs. \ref{eq1} and \ref{eq5}, for $i=1$, and using eqs. \ref{eq3} }
\\&&&\hspace{50pt}\text{and substituting $j=i-1$ for $i>1$}\\
& &=&g_\#(T)-f_\#(T)-\Phi_{n-1}\partial T \hspace{10pt}\text{by definition of $\partial( T)$}
 \end{align*}
It can be shown that $\Phi$ maps degenerate digital $n$-cubes in $\Xk$ to degenerate digital $(n+1)$-cubes in $\Yl$, inducing a homomorphism $\varphi_n:dC_n(X)\to dC_{n+1}(Y)$. If we choose $T$ to be a non-degenerate $n$-cycle, \textit{i.e.} $T\in dZ_n(X)$, then we get $g_\#(T)-f_\#(T) \in dB_n(Y)$. Therefore in $dH_n(Y)$ we have, \begin{equation*}
[g_\#(T)-f_\#(T)]=g_\ast([T])-f_\ast([T])=0\textit{\hspace{10pt}}\Rightarrow\textit{\hspace{10pt}} g_\ast=f_\ast
\end{equation*} \qed
\end{proof}
\begin{cor}
If $\Xk$ and $\Yl$ be homotopically equivalent digital images, then $dH_n(X)\approx dH_n(Y)$.
\end{cor}
\begin{proof}
Follows from Proposition \ref{prop_homotopy}, and functoriality of $dH_n$.
\qed
\end{proof}

\begin{ex}
A digital image is said to be $\kappa$-contractible \cite{Boxer_99}, if its identity map is $(\kappa,\kappa)$-homotopic to a constant function $c_p$ for some $p\in X$.
For a $\kappa$-contractible digital image $\Xk$, one can compute the homology groups using Propositions \ref{prop_Dim} and \ref{prop_homotopy} as $
 dH_n(X) = 
  \begin{cases} 
   \mathbb{Z},  & \text{if }n=0 \\
   0,        & \text{otherwise,}
  \end{cases}
$
because a $\kappa$-contractible digital image is homotopy equivalent to a point \cite{Boxer_99}.

\end{ex}
\section{Digital Hurewicz theorem}\label{Sec_Hure}

\begin{lem}\label{lem_Hure}
Let $(X,p,\kappa)$ be a digital image with basepoint $p$ and $\kappa$-adjacency relation and ${\rm{\Pi}}^\kappa_1(X,p)$ be the fundamental group. 
Then there is a homomorphism
$
\phi:\piX\to dH_1(X)$ given by  
$[ f]_{_{\rm\Pi}}\mapsto \left[\sum_{j=1}^m f_j\right],
$
where $\sum_{j=1}^m f_j$ is the subdivision of $\kappa$-loop $f$. 

\end{lem}
\begin{proof}
\text{\textit{Well-defined}:} We need to show that $\phi$ is a well-defined. Consider $\kappa$-loops $f$ and $g$ of lengths $m_1$ and $m_2$, respectively, both based at point $p\in X$ such that $[ f]_{_{\rm\Pi}}=[ g]_{_{\rm\Pi}}\in\piX$. 
Now $f$ and $g$ are in the same loop class implies that there 
are trivial extensions $f'$ and $g'$ of $f$ and $g$, respectively 
such that there exists a homotopy $H:[0,m]_\mathbb{Z}\times [0,M]_\mathbb{Z}\to X $ from $f'$ to $g'$ that holds the end points fixed. 
Subdivide $H$ into digital $2$-cubes ${j,k}:I^2\to X$ defined as $(s,t)\mapsto H(j+s-1,k+t-1)$, for $j\in [1,m]_\mathbb{Z}$ and $k\in [1,M]_\mathbb{Z}$ (see Figure \ref{fig_homot_Hure}).
 We shall show that the boundary $\partial\left(\sum_{j,k} H_{j,k}\right)$ is equal to the difference of     $\sum_{j=1}^{m_1}f_j$ and $\sum_{j=1}^{m_2}g_j,$ 
 which implies that the classes of these subdivisions are equal in the homology group $dH_1(X)$. 
 Before computing $\partial\left(\sum_{j,k} H_{j,k}\right)$, note that the following equations hold:
\begin{equation}\label{note2}
\sum_{k=1}^MA_1H_{1,k}=\sum_{j=1}^Mf'_j=\sum_{j=1}^{m_1}f_j
\hspace{10pt}\text{and}\hspace{10pt}
\sum_{k=1}^MB_1H_{m,k}=\sum_{j=1}^Mg'_j=\sum_{j=1}^{m_2}g_j
\end{equation}
The only difference between $f$ and its trivial extension $f'$ is that $f'$ pauses more frequently for rest than $f$ and whenever a path pauses for rest, its subdivision is trivial at that point in $dC_1(X)$ (being degenerate in $dQ_1(X)$). Further, it can be noted that:
\begin{align}\label{note3}\nonumber
&A_1H_{j,k}=B_1H_{j-1,k},\hspace{5pt} 
j\in[2,m]_\mathbb{Z},k\in[1,M]_\mathbb{Z}\\
&\text{and}\hspace{10pt}
A_2H_{j,k}=B_2H_{j,k-1},\hspace{5pt} 
j\in[1,m]_\mathbb{Z},k\in[2,M]_\mathbb{Z}
\end{align}
\begin{equation}\label{note4}
A_2H_{j,1}=B_2H_{j,M}=c_p, \hspace{8pt} j\in [1,m]_\mathbb{Z},
\end{equation}	
where $c_p$ is the constant path of length 1 at basepoint $p\in X$. 
Using eqs. \ref{note2} to \ref{note4}, it can be shown that $\partial\left(\sum_{j,k} H_{j,k}\right)=\sum_{j=1}^{m_2}g_j-\sum_{j=1}^{m_1}f_j\in  dC_1(X)$. This proves that $\phi$ is well-defined.\\
\begin{figure}[t]
\begin{multicols}{2}\centering
\includegraphics[scale=.3]{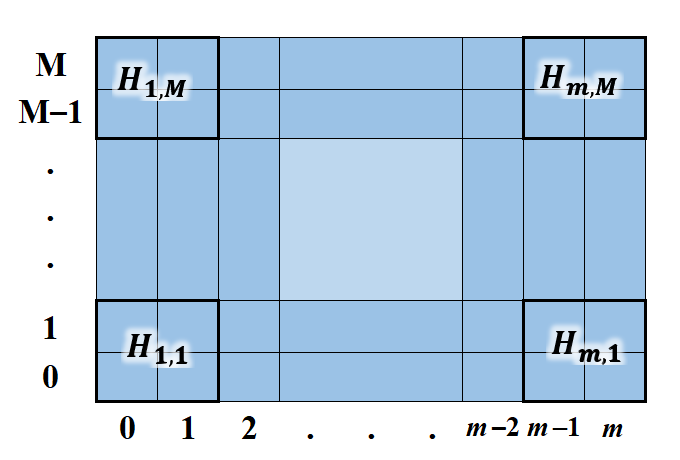}\\
\includegraphics[scale=.3]{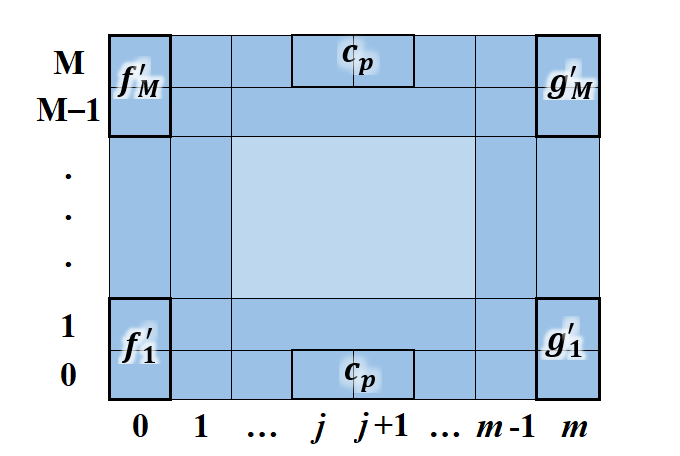}
\end{multicols}
\begin{multicols}{2}
\begin{center}
(a)\\(b)
\end{center}
\end{multicols}
\caption{Domain of  $H$ (proof of Lemma \ref{lem_Hure})}
(a) Subdivision of $H$ into digital $2$-cubes $H_{jk}$ (b) digital $1$-cubes involved in $\partial (H_{jk})$ 
\label{fig_homot_Hure}
\end{figure}
\text{\textit{Homomorphism}:} 
Consider $\kappa$-loops $f$ and $g$ of lengths $m_1$ and $m_2$, respectively, both based at point $p\in X$. Then
\begin{align*}
\phi([ f]_{_{\rm\Pi}}\ast[ g]_{_{\rm\Pi}})&=\phi([ f\ast g]_{_{\rm\Pi}})
				=\left[\sum_{j=1}^{m_1+m_2}(f\ast g)_j\right]
				=\left[\sum_{j=1}^{m_1}(f\ast g)_j+\sum_{j=m_1+1}^{m_1+m_2}(f\ast g)_j\right]\\
				&=\left[\sum_{j=1}^{m_1}f_j+\sum_{j=1}^{m_2}g_j\right]
				=\left[\sum_{j=1}^{m_1}f_j\right]+\left[\sum_{j=1}^{m_2}g_j\right]
				=\phi([ f]_{_{\rm\Pi}})+\phi([ g]_{_{\rm\Pi}})
\end{align*}\qed


\end{proof}
We say that the map $\phi$ defined in Lemma \ref{lem_Hure} is \textit{Digital Hurewicz map}.
\begin{lem}\label{lems_Hurewicz}
Let $\Xk$ be a digital image.
\begin{enumerate}
\item \label{lems_Hurewicz_1} Consider a digital  $1$-cube $T\in dC_1(X)$ and let $\overline{T}$ denote the `reverse' of $T$, \textit{i.e.} $\overline{T}\in dC_1(X)$, $\overline{T}(t)=T(1-t)$. Then class of $T+\overline{T}$ is trivial in $dH_1(X)$.
\item \label{lems_Hurewicz_2} Consider digital  $2$-cube $T\in dC_n( X)$  and define $\kappa$-paths $T_0, T_1, T_2$ and $T_3$ to be $A_1T, A_2T,B_1T$ and $B_2T$, respectively. Then there is a trivial extension of $T_0$ homotopic to $T_1\ast T_2\ast\overline{T_3}$. 
\end{enumerate}
\end{lem}
\begin{proof}\textcolor{white}{dd}
\begin{enumerate}
\item Let $S:I^2\to X$ be a basis element of $dC_2(X)$ defined as $S(t,0)=T(t)$ and $S(t,1)=T(0)$, for $t=0,1$ (see Figure \ref{fig_boundary2cube}(a)). Note that the back $1$-face $B_1S=\overline{T}$ (see Figure \ref{fig_boundary2cube}(a)) and thus the boundary $\partial S=T+\overline{T}$ in $dC_1(X)$ making the class of $T+\overline{T}$ trivial in $dH_1(X)$.
\item Consider the homotopy $H$ defined as $H:[0,3]_\mathbb{Z}\times I\to X$ as\\
$H(0,0)=T_1(0)$,
$H(1,0)=T_2(0)$,
$H(2,0)=T_3(1)$,
$H(3,0)=T_3(0)$,\\
$H(0,1)=H(1,1)=T_0(0)$,
$H(2,1)=H(3,1)=T_0(1)$ (see Figure \ref{fig_boundary2cube}(b) and (c)).\\
Clearly, $H(t,0)=T_1\ast T_2\ast\overline{T_3}(t)$ and $H(t,1)$ is a trivial extension of $T_0$.
\end{enumerate}
\qed
\end{proof}
\begin{figure}[h]
\setlength{\columnsep}{-1.15cm}
\begin{multicols}{3}\centering\noindent
\hspace{-40pt}\includegraphics[scale=.25]{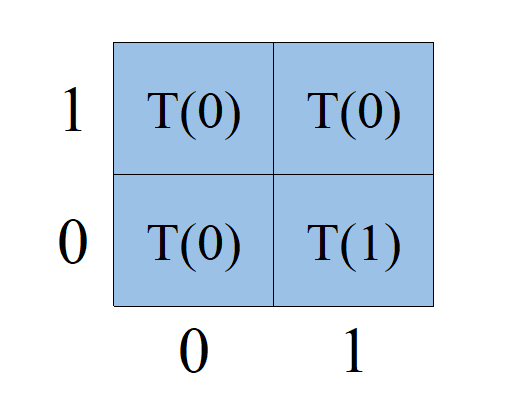}\\
\hspace{-40pt}\includegraphics[scale=.25]{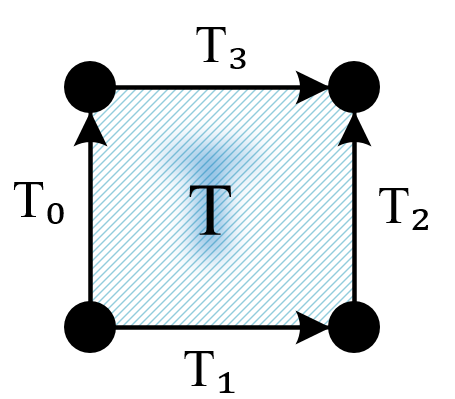}\\
\includegraphics[scale=.25]{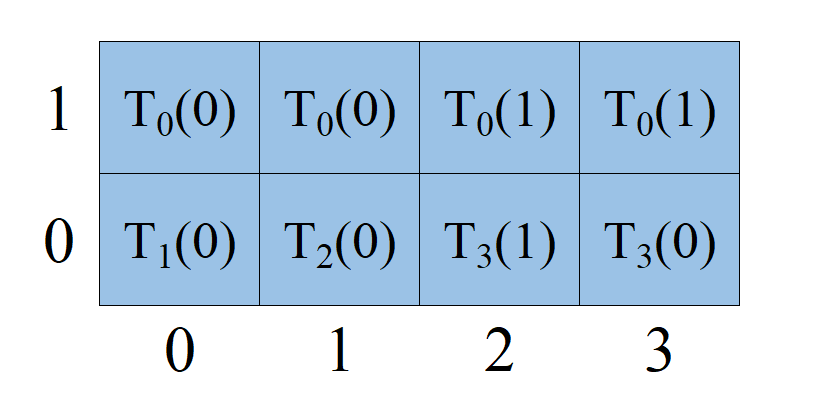}
\end{multicols}\vspace{-20pt}
\begin{multicols}{3}\centering
\hspace{-40pt}(a)\\\hspace{-40pt}(b)\\(c)
\end{multicols}\vspace{-10pt}
\caption{(a) Digital $2$-cube $S$, and (b) faces of digital $2$-cube $T$, (c) Homotopy $H$ (proof of Lemma \ref{lems_Hurewicz})}
{(a) Domain of $S$ with images labeled on each pixel (b) Schematic representation of $T$  (c) Domain  of $H$ with images labeled on each pixel}\label{fig_boundary2cube}
\end{figure}
The following Lemma (quoted from \cite{AT_Rotman} with some minor changes) is required in the proof of digital Hurewicz theorem (Theorem \ref{thm_Hure}). 

\begin{lem}{\text{\emph{Substitution Principle} }}\label{Lem_Substitution}~\\
Let $F$ be a free Abelian group with basis $B$, let $x_0,x_1,\ldots,x_N$ be a list of elements in $B$, possibly with repetitions and assume that $\sum_{i=0}^km_ix_i=\sum_{i=k+1}^{N}m_ix_i$, where $m_i\in\mathbb{Z}$ and $0\leq k<N$. 
If $G$ is any Abelian group and $y_0,y_1,\ldots,y_N$ is a list of elements in $G$ such that $x_i=x_j \Rightarrow y_i=y_j$, then $\sum_{i=0}^km_iy_i=\sum_{i=k+1}^Nm_iy_i$ in $G$.
\end{lem}
\begin{proof}
Define a function $\eta:B\to G$ with $\eta (x_i)=y_i$ for all $i=1,2,\ldots, N$ and $\eta (x)=0$, otherwise ($\eta$ is well-defined because of the given hypothesis). Extend the map $\eta$ by linearity to $\eta :F\to G$. Thus $0=\eta \left(\sum_{i=0}^km_ix_i-\sum_{i=k+1}^Nm_ix_i\right)=\sum_{i=0}^km_iy_i-\sum_{i=k+1}^Nm_iy_i$.
\qed
\end{proof}
\newpage
\begin{thm}\label{thm_Hure}\text{\emph{Digital Hurewicz Theorem}}\\
If $\Xk$ is a $\kappa$-connected digital image with $p\in X$ then the digital Hurewicz map (defined in Lemma \ref{lem_Hure}) is surjective with $ker\phi$ as commutator subgroup of the fundamental group $\piX$. Hence, Abelianized Fundamental group is isomorphic to $dH_1(X)$.
\end{thm}
\begin{proof}
\text{\textit{Surjectivity}:}
Consider $[z]\in dH_1(X)$, with $z=\sum_{i=0}^m n_iT_i$, where $T_i:I\to X$ is a non-degenerate digital $1$-cube, for all $i$. 
Though $n_i\in\mathbb{Z}$, we can assume, without loss of generality,  that $n_i=1, \forall i$, for the following reason: If $n_i=0$, no contribution is made to $z$ by $n_iT_i$ and if $n_i<0$ then we can replace $n_iT_i$ by $-n_i\overline{T_i}$ without changing the class $[z]$, using Lemma \ref{lems_Hurewicz}(\ref{lems_Hurewicz_1}). Thus we can assume $n_i>0, \forall i$, but then 
each $n_iT_i$ can be written as $T_i+T_i+\cdots +T_i$ ($n_i$ terms). Therefore, $z=\sum_{i=0}^m T_i$.  
Since $z$ is a cycle, we have 
\vspace{-2pt}
\begin{equation}\label{eq_Ti10}
\partial z=\partial \left(\sum_{i=0}^mT_i\right)=0 \hspace{10pt}\Rightarrow\hspace{10pt} \sum_{i=0}^m(B_1T_i-A_1T_i)=0.
\end{equation}
For every $i\in [0,m]_\mathbb{Z}$, there exists $j\in [0,m]_\mathbb{Z}$ and $B_1T_i=A_1T_j$,  
so that the sum in eq. \ref{eq_Ti10} is 0, but $i\neq j$, because in case $i=j$, $T_i$ would be degenerate. 
Let $\rho$ be the permutation on elements of $[0,m]_\mathbb{Z}$, satisfying the condition that $A_1T_{\rho(i+1)}=B_1T_{\rho(i)}$ for all $i\in[0,M]_\mathbb{Z}$, where arguments of $\rho$ are read $\mod (M+1)$.  
We can take product of $\kappa$-paths $T_{\rho(i)}$ to get a $\kappa$-loop $\prod_{i=0}^mT_{\rho(i)}$ based at point $T_{\rho(0)}(0)\in X$.
Since the digital image $\Xk$ is $\kappa$-connected, we can take $\kappa$-path $\sigma$ from $p$ to $T_{\rho(0)}(0)$. 
We get:
\begin{align*}
\phi\left(\left[\sigma\ast\prod_{i=0}^mT_{\rho(i)}\ast\overline{\sigma}\right]_{\rm\Pi}\right)
					&=\left[\sum_{l=1}^M\sigma_l+\sum_{i=0}^mT_{\rho(i)}+\sum_{l=1}^M\overline\sigma_l\right]\\
					&=\left[\sum_{l=1}^M\sigma_l+\sum_{i=0}^mT_{\rho(i)}-\sum_{l=1}^M\sigma_l\right],
					\text{ using Lemma \ref{lems_Hurewicz}(\ref{lems_Hurewicz_1})}\\
					&=\left[\sum_{i=0}^mT_i\right]
					=[z].
\end{align*}
\text{\textit{Kernel of $\phi$}:}
Let $\rm\Pi '$ denote the commutator subgroup of $\piX$ and $\overline{\rm\Pi}$ denote the Abelianized fundamental group, \textit{i.e.} $\overline{\rm\Pi}$ is the quotient group $\piX$ modulo the commutator subgroup $\rm\Pi '$.    
Since $dH_1(X)$ is an Abelian group, $\rm\Pi '\subset ker\phi$.  We claim that the reverse inequality also holds. 
Consider a $\kappa$-loop $f$ of length $m$ such that $[ f]_{_{\rm\Pi}}\in ker\phi$. 
It suffices to show that $\llbracket  f \rrbracket $ is identity in $\overline{\rm\Pi}$, 
where $\llbracket  f \rrbracket\in\overline{\rm\Pi}$.
Since $\phi([ f]_{_{\rm\Pi}})=0$, the cycle $\sum_{j=1}^m f_j$ lies in the boundary group $dB_1(X)$, 
\textit{i.e.} there is $\sum_{i=1}^Nn_iT_i\in dC_2(X)$ such that $\sum_{j=1}^m f_j = \partial (\sum_{i=1}^Nn_iT_i) $, 
where $n_i\in\mathbb{Z}$ and $T_i:I^2\to X$ are digital $2$-cubes. We assume without loss of generality that $n_i=1, \forall i$.
Lets denote $A_1T_i,$ $A_2T_i,$ $B_1T_i$ and $B_2T_i$ 
as $T_{i0}, $ $T_{i1},$ $T_{i2}$ and $T_{i3}$, respectively, for $i\in [1,N]_\mathbb{Z}$. We get
\begin{equation}\label{eq_f_j}
\sum_{j=1}^mf_j= \sum_{i=1}^M(-T_{i0}+T_{i2}+T_{i1}-T_{i3})
\end{equation}
This equation has basis elements of the free Abelian group $dC_1(X)$ on both sides. We shall apply substitution principle (Lemma \ref{Lem_Substitution}), to obtain an analogous equation in $\overline{\rm\Pi}$. We need for each term in eq. \ref{eq_f_j}, an element in $\overline{\rm\Pi}$, satisfying the hypothesis of substitution principle.
For each $x\in X$, choose a $\kappa$-path from $p$ to $x$, denoted by $\beta_x$, such that for the base point $p$, $\beta_p = c_p$ is a constant $\kappa$-path at $p$.
For each $j\in [0,m]_\mathbb{Z}$,  define $\kappa$-loops,   $L'_j=\beta_{f(j-1)}\ast f_j \ast \overline{\beta_{f(j)}}$ based at $p$ corresponding to each $f_j$ (see Figure \ref{fig_paths_loop}(a)).
Similarly, define  $\kappa$-loops $L_{iq}=\beta_{T_{iq}(0)}\ast T_{iq}\ast\overline{\beta_{T_{iq}(1)}}$  based at $p$, corresponding to each $T_{iq}$ (see Figure \ref{fig_paths_loop}(b)). 
We get the following in $\piX$:
\begin{align}\nonumber
&\left[\overline{L_{i0}}\ast L_{i1}\ast L_{i2}\ast \overline{L_{i3}}\right]_{_{\rm\Pi}} \\\nonumber
				&=[\beta_{T_{i0}(1)}\ast\overline{ T_{i0}}\ast\overline{\beta_{T_{i0}(0)}}\ast\beta_{T_{i1}(0)}\ast T_{i1}\ast\overline{\beta_{T_{i1}(1)}}\ast \beta_{T_{i2}(0)} \ast T_{i2}\ast\overline{\beta_{T_{i2}(1)}}\ast\beta_{T_{i3}(1)}\ast\overline{T_{i3}}\ast\overline{\beta_{T_{i3}(0)}}]_{_{\rm\Pi}} \\\nonumber
				&=\left[\beta_{T_{i0}(1)}\ast\overline{ T_{i0}}\ast T_{i1}\ast T_{i2}\ast\overline{T_{i3}}\ast\overline{\beta_{T_{i3}(0)}}~\right] _{_{\rm\Pi}}
\\\nonumber
				&=\left[\beta_{T_{i0}(1)}\ast\overline{ T_{i0}}\ast T_{i0}\ast\overline{\beta_{T_{i3}(0)}}\right]_{_{\rm\Pi}}
				\text{,\hspace{10pt} using Lemma \ref{lems_Hurewicz}(\ref{lems_Hurewicz_2})} \\
				&=\left[\beta_{T_{i0}(1)}\ast\overline{\beta_{T_{i3}(0)}}\right] _{_{\rm\Pi}}
				=\left[c_p\right]_{_{\rm\Pi}} \label{eq_const_p}
\end{align}\normalsize
Second equality above follows because $T_{i0}(0)=T_{i1}(0)\Rightarrow\beta_{T_{i0}(0)}=\beta_{T_{i1}(0)}$, $T_{i1}(1)=T_{i2}(0)\Rightarrow\beta_{T_{i1}(1)}=\beta_{T_{i2}(0)}$, $T_{i2}(1)=T_{i3}(1)\Rightarrow\beta_{T_{i2}(1)}=\beta_{T_{i3}(1)}$ and $T_{i3}(0)=T_{i0}(1)\Rightarrow\beta_{T_{i3}(0)}=\beta_{T_{i0}(1)}$ (see Figure \ref{fig_paths_loop}(b)) and for any $\kappa$-path $\varrho$, the loop $\varrho\ast\overline{\varrho}$ is homotopic to constant loop at ${\varrho(0)}$ (see Theorem 4.13 in \cite{Boxer_99}).
\begin{figure}
\begin{multicols}{2}\centering
\includegraphics[scale=.4]{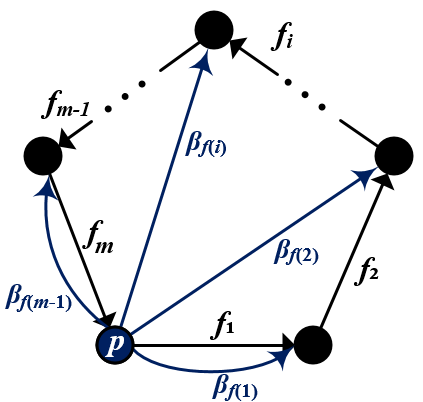}\\
\includegraphics[scale=.4]{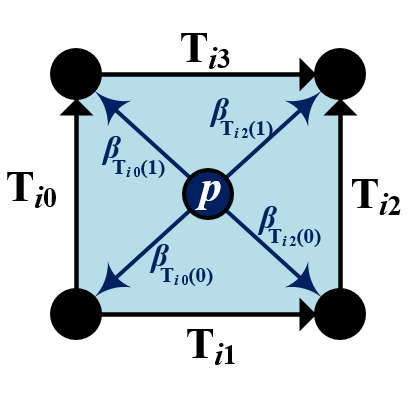}
\end{multicols}
\begin{multicols}{2}\begin{center}
(a)\\(b)
\end{center}
\end{multicols}
\caption{Schematic representation of paths $\beta_x$ (proof of Theorem \ref{thm_Hure})}
{Paths $\beta_x$ are shown in blue color (a) from $p$ to $T_{i0}(1)$, $T_{i1}(0)$, $T_{i2}(0)$ and $T_{i3}(1)$, and (b) from $p$ to $f(j),j\in [0,m-1]_\mathbb{Z}$.}\label{fig_paths_loop} 
\end{figure}

Similarly,  $\left[\prod_{j=1}^m \beta_{f(j-1)}\ast f_j \ast \overline{\beta_{f(j)}}\right]_{_{\rm\Pi}} =\left[ \prod_{j=1}^m f_j\right]_{_{\rm\Pi}}=[f]_{_{\rm\Pi}}$ in $\piX$, because $\beta_{f(0)}=\overline{\beta_{f(m)}}$ is the constant path $c_p$ at $p$.  
Therefore, we get the following in $\overline{\rm\Pi}$,
\begin{align*}
\llbracket f\rrbracket 
			&=\left\llbracket \prod_{j=1}^m f_j\right\rrbracket
			=\left\llbracket \prod_{j=1}^m \beta_{f(j-1)}\ast f_j \ast \overline{\beta_{f(j)}}\right\rrbracket \\ 
			&=\left\llbracket \prod_{i=1}^M \overline{L_{i0}}\ast L_{i1}\ast L_{i2}\ast \overline{L_{i3}}\right\rrbracket,
\end{align*}  
by applying substitution principle (Lemma \ref{Lem_Substitution}) to eq. \ref{eq_f_j} for the free Abelian group $dC_1(X)$  and the multiplicative Abelian group $\overline{\rm\Pi}$. Using eq. \ref{eq_const_p}, $\llbracket f\rrbracket$ is trivial in $\overline{\rm\Pi}$ and $[f]_{_{\rm\Pi}}\in\rm\Pi '$. Therefore, the kernel of the digital Hurewicz map is the commutator of $\piX$, and $\overline{\rm\Pi}\approx dH_1(X)$, using first isomorphism theorem of groups.
\qed
\end{proof}

\section{Relative Homology and Excision}\label{Sec_RelExc}


For a digital image  $\Xk$  and $A\subset X$, $(A,\kappa)$ is a digital image in its own right.
Let $((X,A),\kappa)$ or briefly, $(X,A)$ denote \textit{digital image pair} with $\kappa$-adjacency. A \textit{map of pairs} $f:(X,A)\to (Y,B)$ between digital image pairs $((X,A),\kappa)$ and $((Y,B),\lambda)$ is a map $f:X \to Y$, with $f(A)\subset B$. We say that $f:(X,A)\to (Y,B)$ is $\kl$-\textit{continuous} if $f:X \to Y$ is $\kl$-continuous.
It can be verified that $\partial_n:dC_n(X)\to dC_{n-1}(X)$ maps $dC_n(A)$ to $dC_{n-1}(A)$. 
If $dC_n(X,A)$ denotes the quotient group $dC_n(X)/dC_n(A)$, then $\partial_n$ induces homomorphism $\partial_n:dC_n(X,A) \to dC_{n-1}(X,A)$ satisfying $\partial_{n-1}\circ\partial_n=0$, and making up a chain complex $(dC_\bullet(X,A),\partial)$, given as:
\begin{equation*}
\xymatrix{
\cdots \ar[r]^{\partial_{n+1}~~~~}  
& dC_n(X,A) \ar[r]^{\partial_{n}~~}
& dC_{n-1}(X,A) \ar[r]^{~~~\partial_{n-1}}
& \cdots 
}
\end{equation*}
Lets denote the homology of this chain complex as $dH_n(X,A)$, \textit{i.e.}
\[ dH_n(X,A)=\frac{ker(\partial_n:dC_n(X,A)\to dC_{n-1}(X,A))}{Im(\partial_{n+1}:dC_{n+1}(X,A)\to dC_{n}(X,A))}.
\] 
We say that $dH_n(X,A)$ is  \textit{$n^{th}$-relative cubical singular homology group of the digital image pair}  $(X,A)$. 
Clearly, $dH_n(X)=dH_n(X,\emptyset)$.


\begin{Def}\label{Def_Int_Cl}
Let $\Xk$ be a digital image. We define operators\newline $Int_\kappa:\mathcal{P}(X)\to\mathcal{P}(X)$ and $Cl_\kappa:\mathcal{P}(X)\to\mathcal{P}(X)$ as follows:
\[ \begin{array}{rrl}
&Int_\kappa(A)&=\{x\in A~\vert~N_\kappa(x,X)\subset A\},\\
&Cl_\kappa(A)&=\{x\in X~\vert~N_\kappa(x,X)\cap A\neq\emptyset\},\\
\text{where }& N_\kappa(x,X)&=\{y\in X~\vert~ x \text{ is } \kappa\text{-adjacent or equal to } y \}.
\end{array}\]
We say that $\ink (A)$ is $\kappa$-interior of $A$ in $\Xk$ and $\clk (A)$ is $\kappa$-closure of $A$ in $\Xk$ and the set $N_\kappa(x,X)$ is neighborhood of  $x$ in $\Xk$.
\end{Def}
Notions similar to above appear in \cite{Kong96} and \cite{Esc12} and also, the $\kappa$-interior and $\kappa$-closure operators defined above are very closely related to \textit{dilation} and \textit{erosion} operators, respectively,  used in \cite{Esc12}. The following proposition shows that these operators  satisfy many relations that are similar to those satisfied by their counterparts in topology. 
\begin{prop}\label{prop_IntCl}
Let $\Xk$ be a digital image, $A,B\subset X$ and $x,y\in X$. Then:
\begin{enumerate}[label={\rm{(\roman*)}}]
\item\label{Clint_subset} $A\subset\clk (A)$, $\ink (A)\subset A$\\
\item\label{Int_eq_cl} $\ink (X-A)=X-\clk (A)$,  $X-\ink (A)=\clk (X-A)$\\
\item\label{clint_inclusion} $A\subset B\Rightarrow\clk (A)\subset \clk (B)$ and $\ink (A)\subset \ink (B)$\\
\item\label{Cl_sub_Int} $X= \ink (A)\cup\ink (B) \Leftrightarrow \clk (X-B)\subset \ink (A)$
\end{enumerate}\end{prop}
\begin{proof}
The proofs are simple and follow easily from Definitions \ref{Def_Int_Cl}.
\end{proof}

The $\kappa$-interior and $\kappa$-closure operators for digital images are not idempotent, \textit{i.e.} $Int_\kappa\circ Int_\kappa\neq\ink$ and $\clk\circ\clk\neq\clk$,
unlike interior and closure operators in topology, as shown in the following example. 
\begin{ex}
Consider the digital image $(X,4)$ and $A\subset X$ shown in Figure \ref{fig_IntCl_Int}(a). The interiors $Int_4(A)$ and $Int^2_4(A)=Int_4( Int_4(A))$  are shown in Fig \ref{fig_IntCl_Int}(b) and (c), respectively, and the closures $Cl_4(A)$ and $Cl_4^2(A)=Cl_4(Cl_4(A))$ in $X$ in Figure  \ref{fig_IntCl_Cl} (a) and (b), respectively. 
Clearly, $Int_4\circ Int_4(A)\neq Int_4(A)$ and $Cl_4\circ Cl_4(A)\neq Cl_4(A)$.
\end{ex}

\begin{figure}[H]
\begin{multicols}{3}\centering
\includegraphics[scale=.40, height=3.5cm]{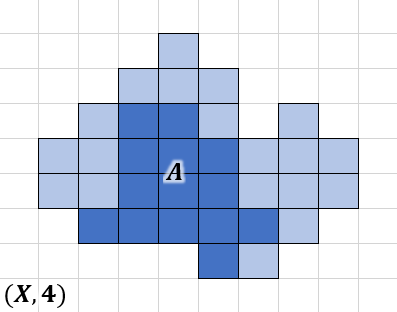}\\
\includegraphics[scale=.40, height=3.5cm]{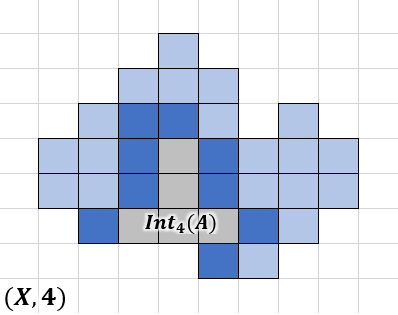}\\
\includegraphics[scale=.40, height=3.5cm]{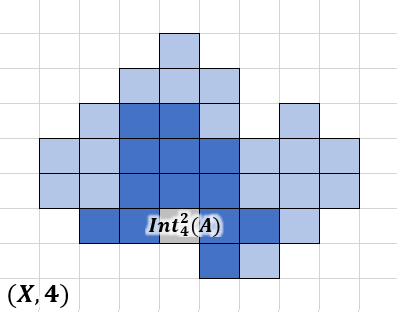}\\
\end{multicols}
\begin{multicols}{3}
\begin{center}
(a)\\(b)\\(c)
\end{center}
\end{multicols}
\caption{(a) Digital image $(X,4)$, its subset $A$ and (b) $Int_4(A)$, (c) $Int^2_4(A)$}
{Digital image $X$, $A$  and interiors are shown in blue, dark blue and grey color, respectively.} \label{fig_IntCl_Int}
\end{figure}
\begin{figure}[H]
\begin{multicols}{2}\centering
\includegraphics[scale=.30, height=3.5cm]{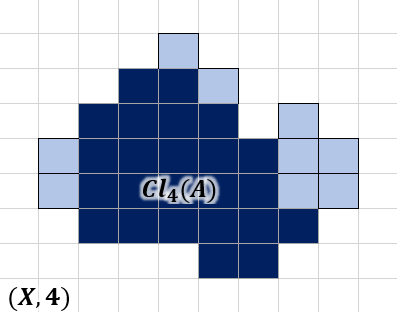}\\
\includegraphics[scale=.30, height=3.5cm]{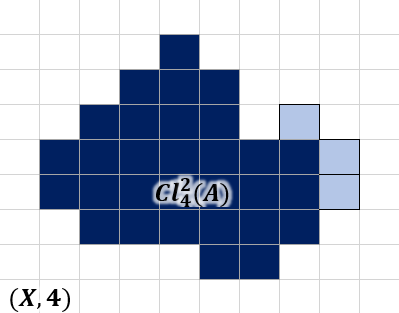}
\end{multicols}
\begin{multicols}{2}
\begin{center}
(a)\\(b)
\end{center}
\end{multicols}

\caption{(a) $Cl_4(A)$ (b) $Cl_4^2(A)$ in $(X,4)$}\label{fig_IntCl_Cl}
{Closures are shown in dark blue color, where digital image $(X,4)$ and $A\subset X$ are shown  in Figure \ref{fig_IntCl_Int}(a).}
\end{figure}
\begin{lem}\label{lem_Excision2}
Let $\Xk$ be a digital image, with subsets $A$ and $B$ such that $X=Int_\kappa(A)\cup Int_\kappa(B)$. Then  for $n\in\{ 0,1,2\}$ and for every digital $n$-cube  $T$, either $Im(T)\subset A$ or $Im(T)\subset B$.
\end{lem}
\begin{proof}
Consider a digital  $n$-cube $T:I^n\to X$ and the following cases for $n\in\{0,1,2\}$: 
\\
\text{Case: ${n=0}$} In this case $Im(T)$ consists of single element, say $x_0$, of $X$. Thus $x_0\in\ink (A)$ or $x_0\in\ink (B)$, implying $Im(T)\subset A$ or $Im(T)\subset B$.
\\
\text{Case: ${n=1}$} In this case, the set $Im(T)\subset X$ comprises two elements, namely, $T(0)$ and $T(1)$. 
We can assume without loss of generality that the element $T(0)\in\ink (A)$. 
By definition of $\ink$ operator, $\kappa$-neighbors of $T(0)$ are in $A$, which implies $T(1)\in A$. Since $\ink (A)\subset A$ (Proposition \ref{prop_IntCl} \ref{Clint_subset}), we get $Im(T)\subset A$.
\\ 
\text{Case: ${n=2}$} In this case, the set $Im(T)\subset X$ comprises at most four distinct elements, namely, $T(0,0)$, $T(0,1)$, $T(1,0)$ and $T(1,1)$. 
We can assume without loss of generality that the element $T(0,0)\in\ink (A)$. By definition of $\ink$ operator, $\kappa$-neighbors of $T(0,0)$ are in $A$, which implies $T(0,1),T(1,0)\in A$. 
Now $T(1,1)$ may or may not lie in $A$. If $T(1,1)\in A$, then $Im(T)\subset A$.
If $T(1,1)\in X-A$, then we claim that $Im(T)\subset B$.  
Our claim follows from the following argument:  From the definition of $\clk$ operator, $T(1,1)\in X-A$ implies that $T(0,1)$ and $T(1,0)$  both lie in
$\clk (X-A)$, which is a subset of $\ink (B)$ by Proposition \ref{prop_IntCl} \ref{Cl_sub_Int}.
Therefore,  
$
 T(0,1), T(1,0)\in\ink (B)\Rightarrow T(0,0)\in B\Rightarrow Im(T)\subset B.
$\qed
\end{proof}
We show in the following example that the above Lemma fails for $n$-cubes with $n>2$.
\begin{ex}
Consider the digital image $(X,4)$   shown in Figure \ref{fig_Int_n2}, where in parts (a) and (b), the subsets $A$ and $B$ of $X$, respectively,  are shown in darker shades of blue. Elements of interiors $Int_4(A)$ and $Int_4(B)$ in $(X,4)$ are shown in part (c) of Figure \ref{fig_Int_n2} as grey-shaded pixels and with double-line borders, respectively. Clearly,   $X=Int_4(A)\cup Int_4(B)$. In these  figures, we have labeled some elements of $X$ as $a,b,c$ and $d$. Define a digital $3$-cube $T$ as follows: $T(0,0,0)=a$,
$T(1,0,0)=T(0,1,0)=T(0,0,1)=b$,
$T(1,1,0)=T(0,1,1)=T(1,0,1)=c$,
$T(1,1,1)=d$. It is clear that neither $Im(T)\subset A$ nor $Im(T)\subset B$. 

\begin{figure}[H]

\begin{multicols}{3}\centering
\includegraphics[scale=.40, height=3.5cm]{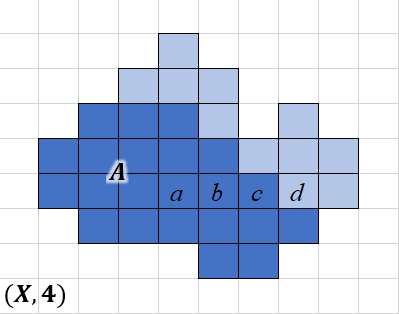}\\
\includegraphics[scale=.40, height=3.5cm]{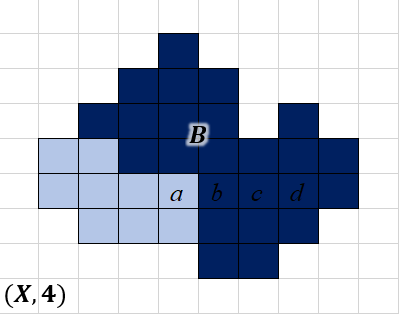}\\
\includegraphics[scale=.40, height=3.5cm]{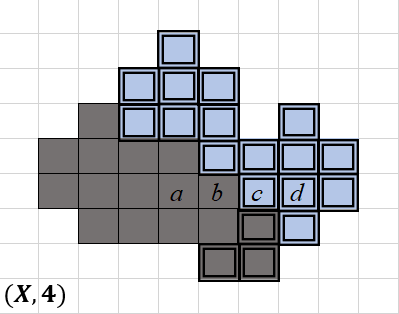}\\
\end{multicols}
\begin{multicols}{3}
\begin{center}
(a)\\(b)\\(c)
\end{center}
\end{multicols}
\caption{(a) Digital image $(X,4)$ with its subset $A$, (b) subset $B\subset X$ (c) interiors $Int_4(A)$ and $Int_4(B)$ in $(X,4)$}
{Parts (a) and (b) show subsets $A$ and $B$ of $X$  in  darker shades of blue, respectively, while part (c) shows the interior $Int_4(A)$ in grey and the interior $Int_4(B)$ with double-line borders.} \label{fig_Int_n2}
\end{figure}

\end{ex}
The following theorem is similar to Excision axiom of homology theory except that it holds only for $n$ less than 2.
\begin{thm}\label{Thm_Excision_nleq2}
Let $\Xk$ be a digital image.
\begin{itemize}[label=\small{\textbullet}]
\item For subsets $A,W\subset X$ such that $\clk (W)\subset\ink (A)$, the inclusion $(X-W,A-W)\to (X,A)$ induces isomorphisms $dH_n(X-W,A-W)\to dH_n(X,A)$, for $n< 2$.
\end{itemize}
Equivalently,
\begin{itemize}[label=\small{\textbullet}]
\item For subsets $A,B\subset X$ such that $X=\ink (A)\cup\ink (B)$, the inclusion $(B,A\cap B)\to(X,A)$ induces isomorphisms $dH_n(B,A\cap B)\to dH_n(X,A)$, for  $n< 2$.\\
\end{itemize}
\end{thm}

\begin{proof} The equivalence of the two statements follows from Proposition \ref{prop_IntCl} \ref{Cl_sub_Int} by taking $B=X-W$, which implies $W=X-B$ and $A-W=A\cap B$.
\\
One can verify that for all $n$,  $dC_n(A)\cap dC_n(B)=dC_n(A\cap B)$ and  for $n\leq 2$, $dC_n(X)=dC_n(A)+dC_n(B)$ using Lemma \ref{lem_Excision2}. 
Furthermore, the map $\frac{dC_n(B)}{dC_n(A)\cap dC_n(B)}\to\frac{dC_n(A)+dC_n(B)}{dC_n(A)}$ induced by inclusion is an isomorphism by second isomorphism theorem of groups.
Therefore we get: 
\begin{align*}
dC_n(B,A\cap B)
		&=\frac{dC_n(B)}{dC_n(A\cap B)}
\approx\frac{dC_n(A)+dC_n(B)}{dC_n(A)}
		=\frac{dC_n(X)}{dC_n(A)}
		=dC_n(X,A),
\end{align*}
where only the second last equality is restricted to $n\leq 2$. It follows that $dH_n(X,A)\approx dH_n(B,A\cap B)$, for integers $n< 2$.
\qed
\end{proof}
Theorem \ref{Thm_Excision_nleq2} is restricted to $n<2$. 
The first of the two versions of Theorem \ref{Thm_Excision_nleq2} states that there is no change in the $n^{th}$-relative homology groups of the digital image pair $(X,A)$, when $n<2$, if we excise out a subset $W$, which is contained `well-inside' $A$. 
In order to extend this idea to higher homology groups ($n\geq 2$), we need the subset $W$ to be contained deeper inside $A$. This can be done by iterative applications of interior and closure operators.
This gives rise to the following definitions and results similar to those in Proposition \ref{prop_IntCl}. 

\begin{Def}\label{Def_Int_Cl_i}
Let $\Xk$ be a digital image and $A\subset X$. We define the operators\newline $Int^i_\kappa:\mathcal{P}(X)\to\mathcal{P}(X)$ and $Cl^i_\kappa:\mathcal{P}(X)\to\mathcal{P}(X)$, for non-negative integers $i$, recursively, as follows: 
\begin{align*}
Int^0_\kappa(A)=A,\hspace{10pt}
&Int^i_\kappa(A)=\ink(\ink^{i-1} (A)),\text{ for positive integer } i,\\
Cl^0_\kappa(A)=A,\hspace{10pt}
&Cl^i_\kappa(A)=\clk(\clk^{i-1} (A)),\text{ for positive integer } i.
\end{align*}
\end{Def}
\begin{prop}\label{prop_IntCl_i}
Let $\Xk$ be a digital image, $A,B\subset X$ and $x,y\in X$. Then:
\noindent
\begin{enumerate}[label={\rm{(\roman*)}}]
\item\label{Clint_subset_i} $\clk^i (A)\subset\clk^{i+1} (A)$, $\ink^{i+1} (A)\subset \ink^i (A)$\\
\item\label{Int_eq_cl_i} $\ink^i (X-A)=X-\clk^i (A)$, $X-\ink^i (A)=\clk^i (X-A)$\\
\item\label{Cl_sub_Int_i} $X= \ink^i (A)\cup\ink^i (B) $ $\Leftrightarrow \clk^i (X-B)\subset \ink^i (A)$
\end{enumerate}
\end{prop}
\begin{proof}
The proofs are simple and follow easily from Definitions \ref{Def_Int_Cl} and \ref{Def_Int_Cl_i}, and Proposition \ref{prop_IntCl}.
\qed
\end{proof}
We give a generalization of Lemma \ref{lem_Excision2}, using Definitions \ref{Def_Int_Cl_i} and Proposition \ref{prop_IntCl_i}.
\begin{lem}\label{lem_Excision_i}
Let $\Xk$ be a digital image, with subsets $A$ and $B$ such that there is a positive integer $i$ with $X=Int^i_\kappa(A)\cup Int^i_\kappa(B)$. Then for $ n\leq i+1$ and for every digital $n$-cube $T$, $Im(T)\subset A$ or $Im(T)\subset B$.
\end{lem}
\begin{proof}
Consider a digital  $n$-cube $T:I^n\to X$, $n\in\{0,1,\ldots,i+1\}$. The set $Im(T)\subset X$ can be partitioned into sets $S_j$ for $j=0,1,\ldots,n$ defined as follows: 
\begin{equation*}
S_j=\{T(x_1,x_2,\ldots,x_n)~\vert~{\rm\Sigma}_{i=1}^n x_i=j\}
\end{equation*}
Note that for  $j\in\{1,2,\ldots,n-1\}$, elements of $S_j$ are $\kappa$-neighbors of elements of $S_{j+1}$ and $S_{j-1}$ and that $S_0$ and $S_n$ are singletons.\\
\text{Case: ${n=0}$} In this case, the partition of $Im(T)$ consists of single set $S_0\subset X$. 
Thus $S_0\subset\ink^i (A)$ or $S_0\subset\ink^i (B)$, implying $Im(T)\subset A$ or $Im(T)\subset B$. 
\\ 
\text{Case: ${0< n< i+1}$} 
 We can assume without loss of generality that the singleton $S_0\subset\ink^i (A)$. 
Then by definition of $\ink^i$ operator, $S_j\subset\ink^{i-j}(A)$, for $j=1,2,\ldots,n$. 
 Thus for all $j$, $S_j\subset A$, since $\ink^i(A)\subset A$ from Proposition \ref{prop_IntCl_i} \ref{Clint_subset_i}. 
Therefore, $Im(T)\subset A$.
\\ 
\text{Case: ${n=i+1}$} Again, we can assume without loss of generality that the set $S_0\subset\ink^i (A)$. From the definition of $\ink^i$, for $j=1,2,\ldots,n-1$,  $S_j\subset\ink^{i-j} (A)$. 
Now $Im(T)-S_n\subset A$ and $S_n$ may or may not lie in $A$. If $S_n\subset A$, then $Im(T)\subset A$, which completes the proof.\\
However, if $S_n\subset X-A$, then we claim that $Im(T)\subset B$, which also completes the proof.  
Our claim follows from the following argument:  
From the definition of $\clk$ operator, $S_n\subset X-A$ implies $S_{n-1}$ is contained in $\clk (X-A)$.
Using Proposition \ref{prop_IntCl_i}, we get the following: 
\[X-A\subset\clk (X-A)\subset\clk^i (X-A)\subset\ink^i (B), \] 
\[\Rightarrow S_{n-1}\subset\ink^i (B)~~
\Rightarrow S_{n-j}\subset\ink^{i-j+1}(B), \text{ for  } j=2,3,\ldots, n~~\Rightarrow Im(T)\subset B.\]
\qed
\end{proof}

\begin{thm}\label{Thm_Excision_like}
\text{\emph{[Excision-like property]}}\\
Let $\Xk$ be a digital image.
\begin{itemize}[label=\small{$\bullet$}]
\item For subsets $A,W\subset X$ such that there is a positive integer $i$, with $\clk^i (W)\subset\ink^i (A)$, the inclusion $(X-W,A-W)\to (X,A)$ induces isomorphisms $dH_n(X-W,A-W)\to dH_n(X,A)$, for integers $ n< i+1$.
\end{itemize}
Equivalently,
\begin{itemize}[label=\small{$\bullet$}]
\item For subsets $A,B\subset X$ such that there is a positive integer $i$, with  $X=\ink^i (A)\cup\ink^i (B)$, the inclusion $(B,A\cap B)\to(X,A)$ induces isomorphisms $dH_n(B,A\cap B)\to dH_n(X,A)$, for integers $ n< i+1$.
\end{itemize}
\end{thm}

\begin{proof}
The equivalence of the two statements follows from Proposition \ref{prop_IntCl_i} \ref{Cl_sub_Int_i} as in the proof of Theorem \ref{Thm_Excision_nleq2}. Rest of the proof is also similar to the proof of Theorem \ref{Thm_Excision_nleq2} except that the equality $(dC_n(A)+dC_n(B))/dC_n(A)=dC_n(X)/dC_n(A)$ holds  for $n\leq i+1$ from Lemma \ref{lem_Excision_i}.\qed
\end{proof}
The following result states the condition under which Excision-like property for $n^{th}$-digital cubical-singular homology holds for all $n$.
\begin{cor}\label{cor_exicision}
Let $\Xk$ be a digital image.
\begin{itemize}[label=\small{$\bullet$}]
\item For subsets $A,W\subset X$ such that $W\subset A$, $\clk(W)=W$ and $\ink(A)=A$, the inclusion $(X-W,A-W)\to (X,A)$ induces isomorphisms $dH_n(X-W,A-W)\to dH_n(X,A)$, for all $ n$.
\end{itemize}
Equivalently,
\begin{itemize}[label=\small{$\bullet$}]
\item For subsets $A,B\subset X$ such that $X=A\cup B$, $\ink(A)=A$ and $\ink(B)=B$, the inclusion $(B,A\cap B)\to(X,A)$ induces isomorphisms $dH_n(B,A\cap B)\to dH_n(X,A)$, for all $ n$.
\end{itemize}
\end{cor}
\begin{proof}
The equivalence of the statements can be shown in a similar way as in the proof of Theorem \ref{Thm_Excision_nleq2}.
Using the hypothesis of first statement, one can show that for all integers $i$, $\clk^i(W)=W$ and $\ink^i(A)=A$, therefore  $\clk^i(W)\subset\ink^i(A)$ also holds for all integers $i$. Rest follows from Theorem \ref{Thm_Excision_like}.
\end{proof}

\section{Digital Homology Theory}\label{Sec_DigHom}
We define \textit{category  of digital-image pairs} $\cat{Dig^2}$ with digital-image pairs as objects and $\kl$-continuous maps of pairs as morphisms. It can be shown that $dH_n(-,-)$ is a functor from $\cat{Dig^2}$ to $\cat{Ab}$ in a similar way as in Proposition \ref{prop_funcHn}. 
\begin{Def}
We say that $\kl$-continuous maps of pairs $f,g:(X,A)\to (Y,B)$ are \textit{$\kl$-homotopic as maps of pairs}, if $H:[0,m]_\mathbb{Z}\times X\to Y$ is $\kl$-homotopy from  $f:X\to Y$ to $g:X\to Y$ and $H(t, A)\subset B,$ $\forall t\in [0,m]_\mathbb{Z}$. 
\end{Def}
\begin{Def}\label{Def_HomTh}
Digital homology theory consists of functors ${dH}_n(-,-)$ from the category of digital image pairs $\cat{Dig^2}$ to the category of Abelian groups $\mathsf{Ab}$ along with natural transformations $\partial_\ast:dH_n(X,A)\to dH_{n-1}(A)$, (where $dH_{n-1}(A,\emptyset)$ is denoted as $dH_{n-1}(A)$) satisfying following axioms:

\begin{description}
\item [\rm{[\textit{Homotopy axiom}]}]~
\newline If $f,g:(X,A)\to (Y,B)$ are homotopically equivalent, then $f_\ast,g_\ast:dH_n(X,A)\to dH_n(Y,B)$ are equal maps.
\item [\rm{[\textit{Exactness axiom}]}]~
\newline For each digital image pair $(X,A)$, and inclusion maps $i:A\hookrightarrow X$ and $j:(X,\emptyset)\hookrightarrow (X,A)$, there is a long-exact sequence:
\begin{equation*}
\xymatrix{
\cdots \ar[r]^{\partial_\ast~~~~}  
& dH_n(A) \ar[r]^{i_\ast}
& dH_n(X) \ar[r]^{j_\ast}
& dH_n(X,A) \ar[r]^{\partial_\ast}
& dH_{n-1}(A) \ar[r]^{~~~~~i_\ast}
& \cdots
}
\end{equation*}
\item [\rm{[\textit{Excision axiom}]}]~\newline 
For  a digital image pair $(X,A)$ and a subset $W\subset A$ such that there is a positive integer $i$ with $\clk^i (W)\subset \ink^i (A)$, the inclusion $(X-W,A-W)\to (X,A)$ induces isomorphism $dH_n(X-W,A-W)\to dH_n(X,A)$ for $0\leq n\leq i+1$. 
\item [\rm{[\textit{Dimension axiom}]}]~\newline 
If $X=\{x_0\}$ is a one-point digital image,
$  dH_n(X) =   0$, for all $n>0$.

\item [\rm{[\textit{Additivity axiom}]}]~\newline 
If $\{(X_\alpha,\kappa)~\vert~\alpha\in\Lambda\}$ is a collection of mutually $\kappa$-disconnected digital images with $X_\alpha\subset\mathbb{Z}^d$ and $\Xk$ is the digital image $X=\bigcup_\alpha X_\alpha$, then $dH_n(X)\approx\bigoplus_\alpha dH_n(X_\alpha)$. 
\end{description}
\end{Def}

\begin{thm}
The relative cubical singular homology groups $dH_n(-,-)$ form a digital homology theory.
\end{thm}
\begin{proof} We prove the axioms of digital homology theory one-by-one:
\begin{description}
\item[\rm{[\textit{Homotopy axiom}]}] 
It can be shown,  using Theorem \ref{prop_homotopy} that if $f,g:(X,A)\to (Y,B)$ are homotopically equivalent, then $f$ and $g$ induce the same map $f_\ast=g_\ast$ from $dH_n(X,A)$ to $dH_n(Y,B)$.
\item[\rm{[\textit{Exactness axiom}]}] 
For a digital image pair $(X,A)$, we have chain complexes $(dC_\bullet(A),\partial)$, $(dC_\bullet(X),\partial)$ and $(dC_\bullet(X,A),\partial)$. We also have chain maps $i_\ast: dC_n(A)\to dC_n(X)$ and $j_\ast:dC_n(X)\to dC_n(X,A)$, 
induced by inclusions $i:A\hookrightarrow X$ and $j:(X,\emptyset)\hookrightarrow (X,A)$. This gives the following short exact sequence of chain-complexes.
\begin{equation*}
\xymatrix{
0 \ar[r]  
& dC_\bullet(A) \ar[r]^{i_\ast}
& dC_\bullet(X) \ar[r]^{j_\ast}
& dC_\bullet(X,A) \ar[r]
& 0
}
\end{equation*}
The above short-exact sequence induces the following long-exact sequence of homology groups: 
\begin{equation*}
\xymatrix{
\cdots \ar[r]^{\partial_\ast~~~~}  
& dH_n(A) \ar[r]^{i_\ast}
& dH_n(X) \ar[r]^{j_\ast}
& dH_n(X,A) \ar[r]^{\partial_\ast}
& dH_{n-1}(A) \ar[r]^{~~~~~i_\ast}
& \cdots
}
\end{equation*}
by zig-zag lemma (\cite{AT_Munkres}, Lemma 24.1). The zig-zag lemma also asserts the existence and uniqueness of the homomorphism $\partial_\ast:dC_n(X,A)\to dC_{n-1}(A)$.
\item[\rm{[\textit{Excision axiom}]}] See Theorem \ref{Thm_Excision_like}.
\item[\rm{[\textit{Dimension axiom}]}] Can be easily proved using Proposition \ref{prop_Dim}.  
\item[\rm{[\textit{Additivity axiom}]}]See Proposition \ref{Prop_DirectSum}. \qed
\end{description}
\end{proof} 
\section{Conclusion}
We have developed cubical singular homology  for digital images as functors from the category of digital images $\cat{Dig}$ to the category of Abelian groups. We showed that the Abelianized fundamental group of digital images developed by Boxer \cite{Boxer_99} is isomorphic to our first homology group. Furthermore, we showed that the sequence of functors satisfy axioms that can be regarded as digital analogue to Eilenberg-Steenrod axioms. 
We also defined digital version of  homology theory.  \\

Singular homology for topological spaces is in general difficult to compute  and the same is true for our case of cubical singular homology for digital  images. More theoretical study is required to make computations possible to some extent. This work can be extended in various directions. Based on our work, cohomology theory for digital images can be developed. Our work is restricted to black-and-white digital images, one might extend this work to develop homology theory for grey-scale and colored digital images.







\end{document}